\documentclass[a4paper,12pt]{amsart}

\usepackage[includehead, twoside=False]{geometry}%
\usepackage{amsmath,amssymb, amsthm,amsxtra}
\usepackage{stmaryrd}

\usepackage[mathscr]{eucal}

\usepackage{bm}
\usepackage[all]{xy}

\usepackage{aliascnt}

\theoremstyle{plain}
\newtheorem{thm}{Theorem}[section]
\newtheorem*{thm*}{Theorem}

\newaliascnt{prop}{thm}
\newaliascnt{cor}{thm}
\newaliascnt{lem}{thm}
\newaliascnt{claim}{thm}
\newaliascnt{defn}{thm}
\newaliascnt{ques}{thm}
\newaliascnt{conj}{thm}
\newaliascnt{fact}{thm}
\newaliascnt{rem}{thm}
\newaliascnt{ex}{thm}
\newtheorem{prop}[prop]{Proposition}
\newtheorem{cor}[cor]{Corollary}
\newtheorem{lem}[lem]{Lemma}
\newtheorem{claim}[claim]{Claim}
\newtheorem*{prop*}{Proposition}
\newtheorem*{cor*}{Corollary}
\newtheorem*{lem*}{Lemma}
\newtheorem*{claim*}{Claim}
\theoremstyle{definition}
\newtheorem{defn}[defn]{Definition}

\newtheorem*{defn*}{Definition}
\newtheorem*{ques*}{Question}
\newtheorem*{conj*}{Conjecture}

\newtheorem*{prob*}{Problem}

\newtheorem{rem}[rem]{Remark}
\newtheorem{ex}[ex]{Example}
\newtheorem*{fact*}{Fact}
\newtheorem*{rem*}{Remark}
\newtheorem*{ex*}{Example}
\aliascntresetthe{prop}
\aliascntresetthe{cor}
\aliascntresetthe{lem}
\aliascntresetthe{claim}
\aliascntresetthe{defn}
\aliascntresetthe{ques}
\aliascntresetthe{conj}
\aliascntresetthe{fact}
\aliascntresetthe{rem}
\aliascntresetthe{ex}

\usepackage[pointedenum]{paralist}
\usepackage{varioref}
\labelformat{equation}{\textnormal{(#1)}}
\labelformat{enumi}{\textnormal{(#1)}}

\setdefaultenum{(a)}{(i)}{(1)}{(A)}
\usepackage[a4paper]{hyperref}

\def\textsectionN~{\textsection{}}

\def\equationautorefname~#1\null{#1\null}
\usepackage{ifthen}
\usepackage{xspace}
\renewcommand\phi{\varphi}
\renewcommand\epsilon{\varepsilon}
\renewcommand\leq{\leqslant}
\renewcommand\geq{\geqslant}

\newcommand\num{\mathbb}\makeatletter
\newcommand{\set}{  \@ifstar{\@setstar}{\@set}}\newcommand{\@setstar}[2]{\{\, #1 \mid #2 \,\}}
\newcommand{\@set}[1]{\{\, #1 \,\}}
\makeatother
\newcommand{\lin}[1]{\langle\, #1 \,\rangle}
\newcommand{\ZZ}{\num{Z}}

\newcommand{\CC}{\num{C}}

\DeclareMathOperator{\ev}{ev}\DeclareMathOperator{\Cone}{Cone}\DeclareMathOperator{\Hilb}{Hilb}\DeclareMathOperator{\im}{im}\DeclareMathOperator{\Pic}{Pic}\DeclareMathOperator{\Sing}{Sing}\DeclareMathOperator{\Loc}{Loc}
\newcommand{\sO}{\ensuremath{\mathscr{O}}}

\newcommand{\Gr}{\mathbb{G}}
\newcommand{\TT}{\mathbb{T}}
\newcommand{\PP}{{\mathbb{P}}}
\newcommand{\Pn}{\PP^n}
\newcommand{\Lxy}{\Loc(R_{xy}')}
\newcommand{\sU}{\mathscr{U}}
\newcommand{\sUs}{\mathscr{U}^{\star}}
\newcommand{\sUU}{\sU \times_{R} \sU}
\newcommand\spcirc{^\circ}
\newcommand{\pr}{\mathrm{pr}}
\newcommand{\Rs}{R^{\star}}
\newcommand{\Lsx}[1][x]{\Loc({\Rs_{#1}}')}
\newcommand{\Pv}[1][n]{(\PP^{#1})\spcheck}
\newcommand{\gen}{\text{:\,general}}
\newcommand{\as}{$(d, \dim Y) = (6, n-3)$ and the formula~\autoref{eq:r-geq-exdim}}

\title[Conics on Fano hypersurfaces]
{Dimension of the space of conics on \\ Fano hypersurfaces}

\author[K.~Furukawa]{Katsuhisa~FURUKAWA}
\address{Graduate School of Mathematical Sciences, 
the University of Tokyo, Tokyo, Japan}
\email{katu@ms.u-tokyo.ac.jp}

\begin{document}

\maketitle

\begin{abstract}
  R. Beheshti showed that, for a smooth Fano hypersurface $X$ of degree $\leq 8$ over the complex number field $\CC$, the dimension of the space of lines lying in $X$ is equal to the expected dimension.

  We study the space of conics on $X$. In this case, if $X$ contains some linear subvariety, then the dimension of the space can be larger than the expected dimension.

  In this paper, we show that, for a smooth Fano hypersurface $X$ of degree $\leq 6$ over $\CC$, and for an irreducible component $R$ of the space of conics lying in $X$, if the $2$-plane spanned by a general conic of $R$ is not contained in $X$, then the dimension of $R$ is equal to the expected dimension.
\end{abstract}

\section{Introduction}

Let $X \subset \Pn$ be a hypersurface of degree $d$ over the complex number field $\CC$. 
We define $R_e (X)$ to be the space of smooth rational curves of degree $e$ in $\PP^n$ lying in $X$,
which is an open subscheme of the Hilbert scheme $\Hilb^{et + 1} (X)$.
The number
\[
{(n+1-d)e+n-4}
\]
is called the \emph{expected dimension} of $R_e (X)$,
where the dimension of $R_e (X)$ at $C$ is greater than or equal to this number
if there exists $C \in R_e (X)$ such that $X$ is smooth along $C$.

The space of rational curves on a \emph{general} Fano hypersurface
have been studied by many authors
(\cite{BV}, \cite{HRS}, \cite{HS}, \cite{St}, \cite{BK}, \cite{RY} in characteristic zero; \cite[V, \textsection{}4]{Ko}, \cite{Fu} in any characteristic).
At least for $e=1,2$, it is well known that
$R_e(X)$ has the expected dimension
if $X$ is a general Fano hypersurface.

On the other hand, it is difficult to know about $R_e(X)$ for \emph{any smooth} $X \subset \Pn$.
For $n$ exponentially large in $d$, T.~D.~Browning and P.~Vishe \cite{BV} showed that the space of rational curves of any degree $e$ on smooth $X$ has the expected dimension.

For any $n$ in the case of $e=1$,
as an answer of the question which was asked by O.~Debarre and J.~de~Jong independently, 
R.~Beheshti \cite{Be1} showed that $R_1(X)$ has the expected dimension
if $X \subset \Pn$ is a smooth Fano hypersurface of degree $d \leq 6$.
J.~M.~Landsberg and C.~Robles \cite{LR} gave another proof for the same degree $d \leq 6$.
Beheshti \cite{Be2} later showed the same statement for $d \leq 8$.

In the case of $e=2$, A.~Collino, J.~P.~Murre, G.~E.~Welters \cite{CMW} studied $R_2(X)$ for a smooth quadric 3-fold $X \subset \PP^4$; in this case, $R_2(X)$ has the expected dimension.
Note that $\Hilb^{2t+1}(X)$ is connected for any smooth hypersurface $X \subset \Pn$
if the expected dimension is positive \cite[Proposition 5.6]{Fu}.

In this paper, we study the dimension of $R_2(X)$ for a smooth Fano hypersurface of degree $d \leq 6$.
Our main result is the following.

\begin{thm}\label{thm:main}
  Let $X \subset \Pn$ be a smooth Fano hypersurface of degree $d \leq 6$ over $\CC$.
  Let $R \neq \emptyset$ be an irreducible component of $R_2(X)$ such that
  \begin{equation}\label{eq:1}
    \text{$\lin{C} \not\subset X$ for general $C \in R$},
  \end{equation}
  where $\lin{C} = \PP^2 \subset \Pn$ is the $2$-plane spanned by $C$.
  Then $\dim (R)$ is equal to the expected dimension $3n-2d-2$.
\end{thm}

The dimension of $R_2(X)$ can be greater than the expected one
when $X$ contains certain linear varieties (see \autoref{thm:ex-Fermat});
this is the reason why we assume the condition~\autoref{eq:1}.
The statement of \autoref{thm:main} does not hold for $d \geq 10$ (see \autoref{thm:ex-sub-Cone});
thus it may need some conditions stronger than \autoref{eq:1} for larger $d$.
\\

The paper is organized as follows.
We assume that $\dim R$ is greater than the expected dimension,
and take $Y := \Loc(R) \subset X$ to be the locus swept out by conics of $R$.
The codimension of $Y$ is $\geq 2$ in $X$ due to
a result of Beheshti \cite{Be2} (see \autoref{thm:Yn-3}).
Then it is sufficient to investigate the case when $(d, \dim Y) = (6, n-3)$.
In \autoref{sec:locus-swept-out-1},
we consider the linear subvariety $\lin{Y} \subset \Pn$ spanned by $Y$,
and show that the codimension of $\lin{Y}$ is $\leq 1$ in $\Pn$ by using projective techniques (\autoref{thm:linY-n-1-pl}).
Let $\TT_xX \subset \Pn$ be the embedded tangent space to $X$ at $x$.
In \autoref{sec:special-point-conic},
considering the subset $\Rs_x \subset R$ consisting of conics $C$
such that $x \in C \subset \TT_xX$, we show that
$\lin{\Loc(\Rs_x)}$ is an $(n-3)$-plane (\autoref{thm:Px-const}),
and show that
$\Loc(\Rs_x)$ is a quadric hypersurface in $\lin{\Loc(\Rs_x)} = \PP^{n-3}$
(\autoref{claim:quad}).
In particular, our problem is reduced to the case $n=d=6$.
Using such quadrics,
we give the proof of \autoref{thm:main} step by step.

\subsection*{Acknowledgments}
The author would like to thank
Professor Hiromichi Takagi
for many helpful comments and advice.
The author was supported by the Grant-in-Aid for JSPS fellows,
No.~16J00404.

\section{The locus swept out by conics}\label{sec:locus-swept-out}
\label{sec:locus-swept-out-1}

We use the following notations.
For a Fano hypersurface $X \subset \Pn$ of degree $d \leq 6$,
we take an irreducible component $R \neq \emptyset$ of $R_2(X)$
such that $\lin{C} \not\subset X$ for general $C \in R_2(X)$.
We denote by $\bar R$ the closure in $\Hilb^{2t+1}(X)$.
Let
\[
\sU := \set*{(C, x) \in R \times X}{x \in C}
\]
be the universal family of $R$,
and let $\pi: \sU \rightarrow R$ and $\ev: \sU \rightarrow X$
be the first and second projections.
For a subset $A \subset R$, we write $\sU_A := \pi^{-1}(A)$ and
$\Loc(A) := \overline{\ev(\sU_A)} = \overline{\bigcup_{C \in A} C} \subset X$.

We write $R_x \subset R$ to be the set of $C \in R$ passing through $x \in X$,
and write $R_{xy} = R_x \cap R_y \subset R$,
the set of $C \in R$ passing through $x,y \in X$.

We set $\lin{S_1\cdots S_m} \subset \Pn$ to be the linear variety spanned by subsets $S_1, \dots, S_m \subset \Pn$.
For example, $\lin{xy} \subset \Pn$ is the line passing through $x,y \in \Pn$,
and $\lin{C} \subset \Pn$ is the $2$-plane spanned by $C$ for a conic $C \subset \Pn$.
\\

The condition~\autoref{eq:1} in \autoref{thm:main} gives the following basic property for $R_{xy}$.

\begin{lem}\label{thm:URxy-Loc-inj}
  Let $R_{xy}'$ be an irreducible component of $R_{xy}$
  such that $\lin{C} \not\subset X$ for general $C \in R_{xy}'$. Then
  \[
  \sU_{R_{xy}'} \rightarrow \Loc(R_{xy}') \subset X
  \]
  is generically finite; moreover a fiber at $z \in \Loc(R_{xy}')$ is of positive dimension
  only if $\lin{xyz} \subset X$.
  In particular, $\dim\Loc(R_{xy}') = \dim(R_{xy}')+1$.
\end{lem}
\begin{proof}
  If $\dim R_{xy}' = 0$, then the assertion follows immediately.
  Assume $\dim R_{xy}' \geq 1$. Then $\dim \Loc(R_{xy}') \geq 2$.
  Take $C \in R_{xy}'$ and $z \in C$ such that $M := \lin{xyz} = \lin{C} \not\subset X$.
  Then $M \cap X$ is a union of finitely many curves.
  Since any conic $\tilde C \in \pi(\ev^{-1}(z) \cap \sU_{R_{xy}'})$ satisfies $\tilde C \subset M$,
  it coincides with a component of $M \cap X$.
  Hence the fiber $\ev^{-1}(z) \cap \sU_{R_{xy}'}$ must be a finite set.
\end{proof}

We set $Y := \Loc(R) \subset X$, the locus swept out by conics $C \in R$,
which is non-linear since $\lin{C} \not\subset X$ for general $C$.
Let us consider the projection
\begin{equation}\label{eq:u2}
  \ev^{(2)}: \sUU \simeq \set*{(C, x, y) \in R \times Y \times Y}{x,y \in C} \rightarrow Y \times Y
\end{equation}
whose fiber at $(x,y) \in Y \times Y$
is isomorphic to $R_{xy}$.
Considering $\sUU \rightarrow R$, we have $\dim \sU \times_{R} \sU = r+2$.
Note that $\ev^{(2)}$ is dominant if and only if
$\Loc(R_x) = Y$ holds for general $x \in Y$.

\begin{rem}\label{thm:Yn-3}
  Assume that $\dim R$ is greater than the expected dimension.
  Then the locus $Y$ is much smaller than $X$. More precisely,
  by a result of R.~Beheshti \cite[Theorem 3.2(b)]{Be2},
  it holds that $\dim Y \leq n-3$.

  We immediately have $\dim Y \leq n-2$;
  this is because if $X = Y = \Loc(R)$ (i.e., $\dim Y = n-1$) in characteristic zero,
  then $R$ has a free curve $C$ and then $R$ must have expected dimension.
  Beheshti's result gives the inequity which is sharper than this.
\end{rem}

\begin{lem}\label{thm:6_n-3}
  If $d \leq 6$ and $r = \dim R$ is greater than the expected dimension,
  then $(d, \dim Y) = (6, n-3)$.
\end{lem}
\begin{proof}
  We have $\dim Y \leq n-3$ due to Beheshti's result as we saw in \autoref{thm:Yn-3}.
  Let $(C, x, y) \in \sUU$ be general.
  Since $\Loc(R_{xy}) \subset Y$,
  it follows from
  the morphism \autoref{eq:u2} and \autoref{thm:URxy-Loc-inj}
  that
  $(r+2-2\dim(Y)) +1 \leq \dim(Y)$.
  Hence $r+3-3\dim(Y) \leq 0$.
  By assumption, $r \geq 3n-2d-1$. Thus
  \[
  3n-2d+2-3\dim(Y) \leq 0.
  \]
  Since $\dim Y \leq n-3$, we have $11-2d \leq 0$; hence $d = 6$.
  Therefore $n-10/3 \leq \dim(Y)$; hence $\dim(Y) = n-3$.
\end{proof}

By the above lemma, let us study the case $(d, \dim Y) = (6, n-3)$,
and assume that $r := \dim R$ is greater than the expected dimension, that is to say,
\begin{equation}\label{eq:r-geq-exdim}
  r \geq 3n - 13.
\end{equation}

\begin{lem}\label{thm:u2surj}
  $\ev^{(2)}$ is dominant. Therefore
  $\Loc(R_x) = Y$ for general $x \in Y$.
\end{lem}
\begin{proof}
  For general $(x,y) \in \im(\ev^{(2)})$,
  we have
  $\dim R_{xy} = (r+2)-\dim(\im(\ev^{(2)}))$,
  which implies
  $\dim R_{xy} +\dim(\im(\ev^{(2)})) = (r+2) \geq 3n-11$.

  Suppose that $\ev^{(2)}$ is not dominant, that is,
  $\dim(\im(\ev^{(2)})) < 2(n-3)$.
  Then
  $\Loc(R_{xy}) \neq Y$, which implies
  $\dim R_{xy} + 1 < \dim Y = n-3$ because of \autoref{thm:URxy-Loc-inj}.
  Then $\dim R_{xy} +\dim(\im(\ev^{(2)})) \leq (n-5)+(2(n-3)-1) = 3n-12$, a contradiction.
\end{proof}

Note that for any $x,y \in Y$
\begin{equation}\label{eq:dimRxy}
  \dim R_{xy} \geq r+2-2\dim(Y) \geq n-5.
\end{equation}

Hereafter we will use several projective techniques
in order to study $Y$.

\begin{rem}\label{thm:gmap}
  We sometimes consider the \emph{Gauss map} of a variety $Z \subset \Pn$,
  which is a rational map
  \[
  \gamma_Z : Z \dashrightarrow \Gr(\dim Z, \Pn),
  \]
  sending a smooth point $x \in Z$ to the embedded tangent space $\TT_xZ \subset \Pn$
  at $x$.
  A general fiber of $\gamma_Z$ is a linear variety of $\Pn$ in characteristic zero (in particular, irreducible).
  The map $\gamma_Z$ is a finite morphism if $Z$ is smooth.
  See \cite[I, \textsection{}2]{Za}.
\end{rem}

We write $\Pv = \Gr(n-1, \Pn)$, the space of hyperplanes of $\Pn$. For a linear subvariety $A \subset \Pn$, we denote by $A^* \subset \Pv$ the set of hyperplanes containing $A$. In addition, for a subset $B \subset \Pn$, we set $\Cone_{A}(B) := \overline{\bigcup_{x \in B} \lin{A, x}}\subset  \Pn$, the cone of $B$ with vertex $A$.
\\

Let us consider the linear variety $\lin{Y} \subset \Pn$ spanned by the locus $Y \subset X$.
The following proposition states
$\lin{Y}$ cannot be so small in $\Pn$.

\begin{prop}\label{thm:linY-n-1-pl}
  Assume \as.
  Then
  $\lin{Y}$ is of dimension $\geq n-1$.
\end{prop}
\begin{proof}
  Since $Y$ is non-linear,
  we immediately have $\dim \lin{Y} > n-3$.
  Now assume that $\lin{Y}$ is an $(n-2)$-plane.
  We need to show two claims.

  \begin{claim}\label{sec:locus-swept-out-2}
    It holds that $Y = \lin{Y} \cap X$
    and $\dim (\Sing Y) \leq 1$.
    In particular, $Y \subset \lin{Y} = \PP^{n-2}$
    is a hypersurface whose degree is equal to $\deg X = d = 6$.
  \end{claim}
  \begin{proof}
    For $x \in \lin{Y} \cap X$, it holds that
    $x \in \Sing (\lin{Y} \cap X)$ if and only if $\lin{Y} \subset \TT_xX$.
    It means that
    $\gamma_X(\Sing (\lin{Y} \cap X)) \subset \lin{Y}^*$
    for the Gauss map $\gamma_X: X \rightarrow \Pv[n]$,
    where $\lin{Y}^* \subset \Pv$ is the set of hyperplanes containing $\lin{Y}$.
    Since $X$ is smooth, $\gamma_X$ is a finite morphism.
    Since $\dim \lin{Y}^* = 1$,
    we have $\dim (\Sing (\lin{Y} \cap X)) \leq 1$.
    If there exists an irreducible component
    $Y' \subset \lin{Y} \cap X \subset \lin{Y} = \PP^{n-2}$ such that $Y' \neq Y$,
    then we have $\dim (Y' \cap Y) \geq n-4 \geq 2$, which is a contradiction
    since $Y' \cap Y \subset \Sing (\lin{Y} \cap X)$. Thus $\lin{Y} \cap X = Y$.
  \end{proof}

  \begin{claim}\label{thm:CnY-neq-0}
    A general $C \in R$ satisfies $C \cap \Sing Y \neq \emptyset$.
  \end{claim}
  \begin{proof}
    Suppose that a general conic $C \in R$ satisfies $C \cap \Sing Y = \emptyset$.
    Since $Y = \Loc(R)$ and the characteristic is zero, $C$ is free in $Y$.
    Then
    $R_2(Y)$ is smooth at $C$ and has the expected dimension $3(n-2) - 2\deg Y - 2$.
    This contradicts that $R \subset R_2(Y)$ is of dimension $> 3n-2d-2$,
    where $\deg Y = \deg X = d$ because of \autoref{sec:locus-swept-out-2}.
  \end{proof}

  From \autoref{thm:CnY-neq-0},
  we may indeed assume
  \begin{equation}\label{eq:CnSempty}
    C \cap S \neq \emptyset
  \end{equation}
  for an irreducible component $S \subset \Sing Y$.
  Note that $\dim S \leq 1$.
  \\

  First we consider the case $n > 6$. 
  Let $C_0 \in R$ be a general conic such that $\lin{C_0} \not\subset X$,
  and take $x,y \in C_0 \setminus S$ be general.
  From \autoref{eq:dimRxy}, we have
  $\dim R_{xy} \geq n-5 \geq 2$.
  Let $R_{xy}' \subset R_{xy}$ be an irreducible component containing $C_0$,
  and take $(R_{xy}')\spcirc \subset R_{xy}'$ to be
  the set of $C$ satisfying $\lin{C} \not\subset X$.
  Then, for any $C \in (R_{xy}')\spcirc$ and $s \in C \cap S$,
  we have $\lin{xys} = \lin{C} \not\subset X$.
  From \autoref{thm:URxy-Loc-inj},
  for the morphism
  \[
  \widetilde{\ev} := \ev|_{\sU_{(R_{xy}')\spcirc}}: \sU_{(R_{xy}')\spcirc} \rightarrow \Loc(R_{xy}'),
  \]
  the preimage $\widetilde{\ev}^{-1}(S)$ is of dimension $\leq 1$.
  Since $\dim R_{xy}' \geq 2$, we have $\pi(\widetilde{\ev}^{-1}(S)) \neq R_{xy}'$,
  which means that $C \cap S = \emptyset$ for general $C \in R_{xy}'$,
  a contradiction.
  \\

  Next we consider the case $n = 6$,
  and complete the proof in the following four steps.
  Note that $\lin{Y} = \PP^4 \subset \PP^6$.
  \\

  \begin{inparaenum}[\noindent\itshape{}Step\,1.]
  \item 
    We show that $S \not\subset \lin{C_0}$
    for general $C_0 \in R$, 
    and also show that
    $S \not\subset M$
    for a general $3$-plane $M \subset \lin{Y}$ containing $C_0$.

    Suppose $S \subset \lin{C_0}$ for general $C_0 \in R$,
    and take $x, y \in Y$ be general points.
    We can assume $y \notin \Cone_x(S)$.   Since $\dim R_{xy} \geq n-5 \geq 1$,
    taking general $C, C' \in R_{xy}$ with $C \neq C'$, we have
    $S \subset \lin{C} \cap \lin{C'} = \lin{xy}$,
    a contradiction.
    
    If $S \subset M$ for a general $3$-plane $M$ containing $C_0$,
    then we can also take another general $\tilde M \neq M$,
    and then $\lin{C_0} = M \cap \tilde M \supset S$, a contradiction.
    \\

  \item
    We consider ${\lin{Y}}\spcheck := \Gr(3, \lin{Y})$,
    the set of $3$-planes in $\lin{Y} = \PP^4$.
    Let
    \[
    W = \set*{(C,M) \in R \times {\lin{Y}}\spcheck}{C \subset M},
    \]
    which is a $\PP^1$-bundle over $R$; in particular,
    $\dim W \geq 6$.
    Let $\pr_2: W \rightarrow {\lin{Y}}\spcheck$ be the projection to the second factor.

    For general $(C_0, M) \in W$,
    take $R^M$ to be an irreducible component of $R \cap R_2(M)$ containing $C_0$.
    We may assume that a general conic $C \in R^M$ satisfies $\lin{C} \not\subset X$.
    Since $R \cap R_2(M) \simeq \pr_2^{-1}(M)$,
    we can assume $\dim R^M \geq 6-\dim(\pr_2(W))$.
    We set the surface
    \[
    Y^M := \Loc(R^M) \subset Y \cap M.
    \]
    
  \item 
    Assume $\dim \pr_2(W) \leq 3$.
    Then we have
    $\dim R^M \geq 3$,
    which implies that
    $R^M \rightarrow \Gr(2, M) = \Pv[3] : C \mapsto \lin{C}$
    is dominant.
    From Step 1,
    $S \cap M$ is a set of finite points. Thus
    $L \cap S = \emptyset$ for a general $2$-plane $L \subset M$.
    Taking a general $C \in R^M$ such that $\lin{C} = L$,
    we find that $C \cap S = \emptyset$,
    which contradicts the condition \autoref{eq:CnSempty}.
    \\

  \item 
    Assume $\dim \pr_2(W) = 4$, that is, $\pr_2(W) = {\lin{Y}}\spcheck$.
    For general $(C_0, M) \in W$, since $M$ is general in ${\lin{Y}}\spcheck$,
    $Y \cap M$ is irreducible.
    Then $Y^M = Y \cap M \subset M = \PP^{3}$, which is a surface
    whose degree is equal to $\deg Y = 6$.

    Since $S \cap M$ is a finite set,
    we may assume that there exists $s \in S \cap M$
    such that $s \in C$ for general $C \in R^M$.
    This implies that $R^M \subset R_s$.
    Considering $\sU_{R^M} \rightarrow Y^M$,
    we find that $\dim (R^M \cap R_x) \geq 1$ for general $x \in Y^M$.
    Since $Y^M$ is surface,
    \[
    Y^M = \Loc(R^M \cap R_x) = \Loc(R^M \cap R_{xs}).
    \]

    For general $C \in R^M$,
    we show that $\lin{C} \cap Y^M$ is scheme-theoretically equal to $C$, as follows.
    Write $(\lin{C} \cap Y^M)_{red} = C \cup \bigcup_i E_i$ with the irreducible components $E_i$'s.
    Take a general $x \in C \setminus \bigcup_i E_i$
    and take a general $y \in E_1 \setminus (C \cup \lin{xs})$.
    Taking the closure
    \[
    \overline{R^M \cap R_{xs}} \subset \Hilb^{2t + 1} (X)
    \]
    and consider the surjective map $\sU_{\overline{R^M \cap R_{xs}}} \rightarrow Y^M$,
    we find $\tilde C \in \overline{R^M \cap R_{xs}}$
    such that $x,y,s \in \tilde C$.
    Then $\lin{\tilde C} = \lin{xys} = \lin{C}$,
    which implies that $\tilde C \subset Y^M \cap \lin{C}$.
    By the choice of $x$, it follows $C \subset \tilde C$.
    Then $C = \tilde C$, which implies $y \in C$, a contradiction.

    Thus $(\lin{C} \cap Y^M)_{red} = C$.
    Suppose that $\lin{C} \cap Y^M$ is non-reduced,
    which means that $C$ is a contact locus on $Y^M$ of $\lin{C}$,
    i.e., $\gamma(C) = \lin{C} \in \Gr(2, M)$
    for the Gauss map
    \[
    \gamma = \gamma_{Y^M}: Y^M \dashrightarrow \Gr(2, M)
    \]
    sending $x \mapsto \TT_xY^M$.
    Then $\dim \gamma(Y^M) = 1$.
    Since $Y^M = \Loc(R^M)$,
    $\lin{C} \in \gamma(Y^M)$ is a general point if so is $C \in R^M$.
    As in \autoref{thm:gmap},
    the general fiber $\gamma^{-1}(\lin{C})$ is a linear variety,
    which contradicts $C \subset \gamma^{-1}(\lin{C})$.
    
    Thus $\lin{C} \cap Y^M = C$ scheme-theoretically.
    This contradicts $\deg Y^M = 6$.
  \end{inparaenum}
\end{proof}

\section{Special point of a conic: the embedded tangent space at the point containing the conic}
\label{sec:special-point-conic}

We use the notations of \autoref{sec:locus-swept-out}.
From \autoref{thm:6_n-3}, we may assume \as.

\begin{lem}
  Let $(C, x) \in \sU$. Then the following holds.
  \begin{enumerate}
  \item 
    $C \not\subset \TT_xX$ if and only if
    $\TT_xX \cap C = \set{x}$

  \item 
    $C \subset \TT_xX$
    if and only if one of the following conditions holds:
    \begin{inparaenum}[\normalfont (i)]
    \item $\lin{C} \subset X$;
    \item $\lin{C} \cap X$ is non-reduced along $C$;
    \item $x \in C \cap E$ for some irreducible component $E \neq C$ of $\lin{C} \cap X$.
    \end{inparaenum}
  \end{enumerate}
\end{lem}
\begin{proof}
  \begin{inparaenum}
  \item
    Since $X \subset \Pn$ is a hypersurface and $C$ is a smooth conic,
    if $C \not\subset \TT_xX$, then
    we have $\TT_xX \cap C
    \subset \TT_xX \cap (\lin{C} \cap C) = \TT_xC \cap C = \set{x}$.

  \item 
    It is sufficient to consider the case when $\lin{C} \not\subset X$.
    If $C \subset \TT_xX$, then $\lin{C} \subset \TT_xX$,
    and then $\lin{C} \cap X$ is singular at $x$. This means that (ii) or (iii) holds.
  \end{inparaenum}
\end{proof}
For a smooth conic $C \subset X$, we always have a point $x \in C$
satisfying $C \subset \TT_xX$, as follows.
Since $\deg(X) = 6$, if (i) and (ii) does not hold, then we have a curve $E \neq C$ in $\lin{C} \cap X$,
and have a point $x \in C \cap E \subset \lin{C} = \PP^2$ as in (iii).

\begin{defn}
  We set $\sUs \subset \sU$ to be an irreducible component of
  $\set*{(C, x) \in \sU}{C \subset \TT_xX}$
  such that $\sUs \rightarrow R$ is dominant,
  and set $R_x^{\star}:= \pi(\ev^{-1}(x) \cap \sUs)$,
  which consists of conics $C \subset X$ such that $x \in C \subset \TT_xX$.
  Note that $\Loc(\Rs_x) \subset \TT_xX$.
\end{defn}

\begin{lem}\label{thm:sUs-neq-sU}
  $\sUs \neq \sU$.
\end{lem}

\begin{proof}
  We consider the Gauss map $\gamma: X \rightarrow \Pv[n]$.
  If $\sUs = \sU$, then it follows from \autoref{thm:u2surj}
  that $Y = \Loc(R_x) = \Loc(\Rs_x)\subset \TT_xX$ for general $x \in Y$,
  and then $\TT_xX = \lin{Y} \subset \Pn$
  because of \autoref{thm:linY-n-1-pl}.
  Then $\gamma(Y) = \lin{Y} \in \Pv[n]$,
  which contradicts that
  $\gamma$ is a finite morphism as we mentioned in \autoref{thm:gmap}.
\end{proof}

\begin{lem}\label{thm:Rxy-Rsx}
  Let $x \neq y \in Y$ satisfy $y \in \TT_xX$. Then $C \subset \TT_xX$ for any $C \in R_{xy}$.
\end{lem}
\begin{proof}
  Take $C \in R_{xy}$ and
  suppose $C \not\subset \TT_xX$.
  Since $\lin{C} \not\subset \TT_xX$,
  $\lin{C} \cap \TT_{x}X$ is the line passing through $x$ and $y$.
  On the other hand, we have $\lin{C} \cap \TT_{x}X = \TT_{x}C$,
  which does not intersect with any point of $C$ except $x$
  since $C$ is a smooth conic,
  a contradiction. Thus we have $C \subset \TT_{x}X$.
\end{proof}

By \autoref{thm:sUs-neq-sU},
we may assume that $\sUs \neq \sU$.
Then $\dim \sUs = \dim R = r$.
We take the projection
\begin{equation*}\label{eq:sUs-x-sU}
  \mu: \sUs \times_R \sU
  \simeq \set*{(C, x, y) \in R \times Y \times Y}{(C, x) \in \sUs, (C, y) \in \sU}
  \rightarrow \ev(\sUs) \times Y,
\end{equation*}
such that $\mu(C, x, y) = (x,y)$, where $\ev: \sU \rightarrow Y$ is the second projection.

Let $(C, x, y) \in \sUs \times_R \sU$ be general.
(Then $x,y \in C \subset \TT_xX$.)
We can take the unique irreducible component
${\Rs_x}' \subset \Rs_x$ 
containing $C$, and
take the unique irreducible component
$R_{xy}' \subset R_{xy}$ 
containing $C$.
(The uniqueness comes from the general choice of $C$.)

\begin{lem}\label{thm:Rxy-sub-Rsx}
  In the above setting, we have
  $R_{xy}' \subset {\Rs_x}'$.
\end{lem}
\begin{proof}
  It follows from \autoref{thm:Rxy-Rsx}.
\end{proof}

Moreover, we have the following key proposition,
where for the projection,
$\im(\mu) \rightarrow \ev(\sUs): (x,y) \mapsto x$,
we also consider the following fiber product
\[
\im(\mu) \times_{\ev(\sUs)} \im(\mu)
\simeq \set*{(x,y,z) \in \ev(\sUs) \times Y \times Y}{(x,y),(x,z) \in \im(\mu)}.
\]
Note that, for an element $(x,y,z)$ of the above set,
there exists conics $C_1, C_2 \subset \TT_xX$
such that $x,y \in C_1$ and $x,z \in C_2$.
The projection
$\im(\mu) \times_{\ev(\sUs)} \im(\mu) \rightarrow Y \times Y$ is dominant.
\begin{prop}\label{thm:Px-const}
  Assume \as.
  Then the following holds.
  \begin{enumerate}
  \item 

    $\Lsx = \Lxy$ and the dimension is $n-4$
    for general $(C, x, y) \in \sUs \times_R \sU$.

  \item 
    $\lin{\Loc({\Rs_x}')}$ is an $(n-3)$-plane of $\Pn$
    for general $(C, x) \in \sUs$.

  \item
    The projection
    \[
    \sUs \times_R \sU \times_R \sU \rightarrow \im(\mu) \times_{\ev(\sUs)} \im(\mu)
    \]
    defined by $(C, x,y,z) \mapsto (x,y,z)$
    is dominant.
    
  \item 
    Let $(C, x, y, z) \in \sUs \times_R \sU \times_R \sU$ be general
    (here, $x,y,z \in C \subset \TT_xX$).
    Then $\Lsx = \Loc(R_{yz}')$,
    where
    $R_{yz}' \subset R_{yz}$ 
    is the unique irreducible component
    containing $C$.

  \end{enumerate}
\end{prop}

From (a) and (b),
we have that
$\Loc({\Rs_x}')$ is a hypersurface of
$\lin{\Loc({\Rs_x}')} = \PP^{n-3}$.
Moreover, as a corollary,
later we will show that $\Loc({\Rs_x}')$ is a quadric hypersurface,
and also will see our problem is reduced to the case of $n=6$
(\autoref{claim:quad}, \autoref{thm:n=6}).

In order to prove this proposition, we show the following two lemmas.

\begin{lem}\label{thm:Loc-st-neq-Y}\label{thm:LRxst=LRxy}
  $\ev(\sUs) \subset Y$ is of dimension $\geq n-5$.
  Hence
  $\Loc(R_x^{\star}) \neq Y$ for general $x \in \ev(\sUs)$.
\end{lem}

\begin{proof}
  Let $(C, x, y) \in \sUs \times_R \sU$
  be general such that $\lin{C} \not\subset X$.
  Considering the morphism $\mu$,
  ${\Rs_x}' \cap R_y \subset R_{xy}$
  is of dimension $\geq r+1-\dim \ev(\sUs)-\dim Y$.
  From \autoref{thm:URxy-Loc-inj},
  $r+2-\dim \ev(\sUs)-\dim Y \leq \dim Y$.
  Hence $\ev(\sUs) \geq r+2 - 2\dim(Y)$,
  where the right hand side is $\geq n-5$.
  
  Since $n \geq 6$, $\ev(\sUs)$ has positive dimension.
  On the other hand,
  we have $\# \set*{x \in \ev(\sUs)}{\Loc(\Rs_x) = Y} < \infty$
  as follows.
  If $\Loc(\Rs_x) = Y$, then $Y \subset \TT_xX$ and then
  $\lin{Y} = \TT_xX$ because of \autoref{thm:linY-n-1-pl}.
  Hence, by the finiteness of the Gauss map of smooth $X$,
  we have the assertion.
\end{proof}

\begin{lem}\label{thm:conex=coney}
  Let $S \subset \Pn$ be a non-linear projective variety.
  Assume $\Cone_xS = \Cone_yS$ for general $x,y \in S$.
  Then $\lin{S} \subset \Pn$ is a $(\dim(S)+1)$-plane.
\end{lem}
\begin{proof}
  Take a general point $x \in S$
  and consider the linear projection
  $\pi_x: \PP^n \rightarrow \PP^{n-1}$ from $x$.
  Then we have
  \[
  \pi_x(S) = \pi_x(\Cone_xS)
  = \pi_x(\Cone_yS) = \Cone_{\pi_x(y)}(\pi_x(S))
  \]
  for general $y \in S$.
  Hence
  $\pi_x(S) = \Cone_{y'}(\pi_x(S))$
  for general $y' \in \pi_x(S)$.
  This means that $\pi_x(S)$ is a $(\dim(S))$-plane.
  Hence $\Cone_xS$ is a $(\dim(S)+1)$-plane, which implies the assertion.
\end{proof}

\begin{proof}[Proof of \autoref{thm:Px-const}]
  \begin{inparaenum}
  \item 
    We may assume $\lin{C} \not\subset X$.
    From \autoref{thm:URxy-Loc-inj} and the formula~\autoref{eq:dimRxy},
    we have $\dim \Lxy \geq n-4$.
    On the other hand, \autoref{thm:LRxst=LRxy} implies
    $\Lsx \leq n-4$.
    Hence
    \autoref{thm:Rxy-sub-Rsx} implies
    $\Lxy = \Lsx$ and the dimension is $n-4$.

  \item 
    From (a),
    for general $y \in \Loc({\Rs_x}')$,
    it follows $\Loc({\Rs_x}') = \Loc(R_{xy}')$,
    where the right hand side is
    the closure of $\bigcup_{C \in R_{xy}'} C \subset \Pn$.
    Since $\Cone_xC = \lin{C} = \Cone_yC$,
    \[
    P_x := \Cone_x\Loc({\Rs_x}')
    = \overline{\bigcup_{C \in R_{xy}'} \lin{C}}
    = \Cone_y\Loc({\Rs_x}').
    \]
    For two general points $y_1, y_2 \in \Loc({\Rs_x}')$,
    we have
    $\Cone_{y_1}\Loc({\Rs_x}') = P_x = \Cone_{y_2}\Loc({\Rs_x}')$.
    Hence the assertion follows from \autoref{thm:conex=coney}.

  \item 
    Since ${\im(\mu) \times_{\ev(\sUs)} \im(\mu)}$ is irreducible,
    it is sufficient to show that,
    for general $x \in \ev(\sUs)$ and for general $y,z \in \Lsx$,
    there exists $C$ such that $(C,x,y,z) \in \sUs \times_R \sU \times_R \sU$.

    First we can take a general $(C_0,x,y) \in \sUs \times_R \sU$ with some $C_0 \in {\Rs_x}'$.
    From (a), we have
    $\Lsx = \Loc(R_{xy}')$.
    Since $z \in \Lsx$ is general,
    we have a general $C \in R_{xy}'$ such that $z \in C$.
    
  \item 
    Consider the projection
    \[
    \pr_{34}: \sUs \times_R \sU \times_R \sU \rightarrow Y \times Y
    \]
    sending $(C,x,y,z) \mapsto (y,z)$.
    Let $F_{yz}$ be an irreducible component of
    the fiber of $\pr_{34}$
    at a general $(y,z) \in Y \times Y$. We identify $F_{yz}$
    and its image in $\sUs$ under
    the projection $\pr_{12}: (C,x,y,z) \mapsto (C,x)$.

    Let us consider
    \[
    {\bigcup_{(C,x) \in F_{yz} \text{: general}} \Lsx}
    \quad \subset {\bigcup_{(C,x) \in F_{yz} \text{: general}} \lin{\Lsx}}.
    \]
    Suppose that the closure of the left hand side is equal to $Y$.
    Then (b) implies
    \[
    \Cone_yY = \Cone_y \overline{\bigcup_{(C,x)} \Lsx} = \overline{\bigcup_{(C,x)} \lin{\Lsx}}.
    \]
    In the same way, $\Cone_zY = \overline{\bigcup_{(C,x)} \lin{\Lsx}}$.
    Hence $\Cone_yY = \Cone_zY$.
    Since $y,z \in Y$ are general, \autoref{thm:conex=coney} implies that
    $\lin{Y} \subset \Pn$ is an $(n-2)$-plane, which contradicts
    \autoref{thm:linY-n-1-pl}.

    Hence the closure of $\bigcup_{(C,x) \in F_{yz} \text{: general}} \Lsx$ is not equal to $Y$.
    Since $\Lsx$ is of codimension $1$ in $Y$,
    it means that $L := \Lsx$ is constant for general $(C,x) \in F_{yz}$.
    Then $\Loc(R_{yz}') = \overline{\bigcup_{(C,x) \in F_{yz} \text{: general}} C} \subset L$.
    Since $\dim \Loc(R_{yz}') \geq n-4$,
    the assertion follows.
  \end{inparaenum}
\end{proof}

\begin{cor}\label{claim:quad}
  Assume \as.
  Then the following holds.
  \begin{enumerate}
  \item 
    For general $(C, x) \in \sUs$,
    $\Lsx \subset \lin{\Lsx} = \PP^{n-3}$ is a quadric hypersurface.
  \item
    The projection to the second factor
    \begin{equation*}      q: \set*{(C, H) \in R \times \Pv}{C \subset H} \rightarrow \Pv
    \end{equation*}
    is dominant. (In particular, a general fiber of $q$
    is of dimension $r-3$.)
  \end{enumerate}
\end{cor}
\begin{proof}
  \begin{inparaenum}
  \item 
    Let $y, z \in \Lsx$ be general.
    Let
    $M \subset \lin{\Lsx}$ be a general $2$-plane
    such that $y,z \in M$.

    From \autoref{thm:Px-const}(c),
    $(C_0, x,y,z)$ is general in $\sUs \times_R \sU \times_R \sU$ with some $C_0$.
    From \autoref{thm:Px-const}(d), it holds that $\Loc(R_{yz}') = \Lsx$.
    We consider
    \[
    R_{yz}' \rightarrow \lin{yz}^*: C \mapsto \lin{C},
    \]
    which is generically finite,
    where $\lin{yz}^* \subset \Gr(2, \lin{\Lsx})$
    is the set of $2$-planes containing the line $\lin{yz}$.
    Since $\dim R_{yz}' \geq n-5 = \dim \lin{yz}^*$,
    this morphism is dominant.

    Thus we can take a general $\tilde M \in \lin{yz}^* \subset \Gr(2, \lin{\Lsx})$ near $M$
    such that $\tilde M = \lin{\tilde C}$ for some $\tilde C \in R_{yz}'$.
    By generality, $\tilde M \cap \Lsx$ is irreducible.
    Hence $\tilde C = \tilde M \cap \Lsx$, which means that $\deg(\Lsx) = 2$.

  \item
    For $\tilde M$ in (a) above,
    we can take a hyperplane $\tilde H \subset \Pn$ containing
    $\tilde M$ as a general element of $\Pv$, and then
    $(\tilde C, \tilde H) \in R \times \Pv$,
    which means that $q$ is dominant.
    
    Considering the projection to the first factor,
    we find that the dimension of the left hand side of $q$
    is $r + n - 3$.
    Thus a general fiber of $q$
    is of dimension $r-3$.
  \end{inparaenum}
\end{proof}

\begin{rem}
  It is known that if a smooth hypersurface in $\Pn$ of degree $> 2$
  contains an $m$-dimensional quadric hypersurface,
  then $m \leq (n-1)/2$.
  Thus \autoref{claim:quad}(a) implies $n \leq 7$.
\end{rem}

\begin{rem}\label{thm:n=6}
  By induction on $n$, 
  it is sufficient to show \autoref{thm:main} in the case of $n = 6$.
  We see the details in the following.
  Let $n > 6$ 
  and let $R \subset R_2(X)$ be an irreducible component.

  Assume \as, i.e.,
  $r := \dim R$
  is greater than the expected one.
  Since $q$ is dominant as in \autoref{claim:quad}(b),
  a fiber of $q$ at general $H \in \Pv$,
  which identified with $R \cap R_2(X \cap H)$,
  is of dimension $r-3$.
  Again since $q$ is dominant, we may take a conic $C \subset H$
  as a general member of $R$.
  Then we may take an irreducible component
  $R'$ of $R_2(X \cap H)$
  containing $C$ such that $\dim R' \geq r-3 \geq 3n-16$.
  Since $C$ satisfies $\lin{C} \not\subset \TT_xX$,
  we have $\lin{C} \not\subset \TT_xX \cap H = \TT_x(X \cap H)$.
  Then a general $\tilde C \in R'$ satisfies
  $\lin{\tilde C} \not\subset \TT_x(X \cap H)$.

  Since $H$ is general, $X \cap H \subset H = \PP^{n-1}$ is smooth.
  Once \autoref{thm:main} is proved for $n-1$, we have a contradiction
  since $\dim R'$ must be equal to the expected dimension $3(n-1)-14$.
\end{rem}

Using the above results and notations, we now prove the main theorem.

\begin{proof}[Proof of \autoref{thm:main}]
  Let $R \neq \emptyset$ be an irreducible component of $R_2(X)$
  satisfying the condition \autoref{eq:1},
  and assume that $r := \dim R$ is greater than the expected dimension.
  From \autoref{thm:6_n-3}, \autoref{claim:quad}(b), and \autoref{thm:n=6},
  we may assume $n=d=6$ and $\dim Y = 3$. Then $r \geq 5$.

  \begin{claim}\label{degY34}
    $\deg(Y) = 3 \text{ or } 4$ and $\dim \lin{Y} \geq 5$.
  \end{claim}
  \begin{proof}
    From \autoref{thm:linY-n-1-pl},
    we have $\dim \lin{Y} \geq 5$.
    In particular, $\deg(Y) \geq 3$.
    Let $H \subset \PP^6$ be a general hyperplane,
    and set $Y' = Y \cap H$.
    It follows from \autoref{claim:quad}(b) again
    that $R \cap R_2(Y')$ is of dimension $\geq 2$.
    It is classically known that,
    if a surface $Y'$ has a $2$-dimensional family of conics, then 
    $Y'$ is projectively equivalent to either the Veronese surface
    $\PP^2 \hookrightarrow \PP^5$
    or its image under linear projections
    (see \cite[p.~130, p.~157]{SR}).
    Hence $\deg(Y) = \deg(Y') \leq 4$.
  \end{proof}

  Our first goal is to show that
  $\lin{Y} \subset \PP^6$ is of dimension $5$.
  Set
  $Q_{yz} \subset Y$
  to be the surface $\Loc(R_{yz}')$
  for general $(C, y, z) \in \sU \times_R \sU$.

  \begin{claim}\label{Lyz=Ltildeyz}
    $Q_{yz}$ is a quadric surface, and
    $Q_{yz} = Q_{\tilde y z}$ holds for general $(\tilde C, \tilde y) \in \sU_{R_{yz}'}$.
  \end{claim}
  \begin{proof}
    Since $(\tilde C,y,z,\tilde y)$ is general in $\sU \times_R \sU \times_R \sU$,
    $R_{yz}'$ (resp. $R_{\tilde yz}'$)
    is the unique irreducible component
    of $R_{yz}$ (resp. $R_{\tilde yz}$)
    containing $\tilde C$.
    We can take $x \in \tilde C$
    such that $(\tilde C, x)$ is general in $\sUs$. Then
    \autoref{thm:Px-const}(d) implies
    $Q_{yz} = \Lsx = Q_{\tilde yz}$,
    which is a quadric
    as in \autoref{claim:quad}.
  \end{proof}

  For general $y, z_1, z_2 \in Y$ (with some $C_i$ such that $(C_i, y, z_i)$ is general in $\sU \times_R \sU$), 
  we write
  \[
  K_{y}^{z_1z_2} := \TT_y Q_{yz_1} \cap \TT_y Q_{yz_2} \subset \TT_yY = \PP^3,
  \]
  whose dimension is $\geq 1$.
  Then $y \in K_{y}^{z_1z_2} \subset \lin{Q_{yz_1}} \cap \lin{Q_{yz_2}}$.
  (Here we do \emph{not} know ``$z_i \in K_{y}^{z_1z_2}$''.)

  \begin{claim}
    $\dim (\lin{Q_{yz_1}} \cap \lin{Q_{yz_2}}) = 1$.
    Hence $K_y^{z_1z_2} = \lin{Q_{yz_1}} \cap \lin{Q_{yz_2}}$.
  \end{claim}
  \begin{proof}
    Suppose $\dim \lin{Q_{yz_1}} \cap \lin{Q_{yz_2}} \geq 2$
    for general $y,z_1,z_2 \in Y$.
    First we take general points $y_0, z_1, z_2 \in Y$.
    By generality, $z_1 \notin \lin{Q_{y_0z_2}}$.
    For general  $y \in Q_{y_0z_2}$,
    we have $\lin{Q_{yz_2}}=\lin{Q_{y_0z_2}}$ because of \autoref{Lyz=Ltildeyz}.
    Consider an open subset $Y\spcirc \subset Y$ containing $y_0$
    such that
    $\dim \lin{Q_{yz_1}} \cap \lin{Q_{yz_2}} \geq 2$
    for $y \in Y\spcirc$.

    Since
    $Y = \overline{\bigcup_{y \in Q_{y_0z_2}\cap Y\spcirc} Q_{yz_1}}$,
    we have
    \[
    \Cone_{z_1}Y
    = \overline{\bigcup_{y \in Q_{y_0z_2}\cap Y\spcirc} \lin{Q_{yz_1}}},
    \]
    where the right hand side contains $Q_{y_0z_2}$; in fact, it contains the $3$-plane $\lin{Q_{y_0z_2}}$
    since each $\lin{Q_{yz_1}}$ satisfies $\dim (\lin{Q_{yz_1}} \cap \lin{Q_{y_0z_2}}) \geq 2$.
    Hence $\Cone_{z_1}Y = \Cone_{z_1}\lin{Q_{y_0z_2}}$, which is a $4$-plane, a contradiction to $\dim \lin{Y} \geq 5$.
  \end{proof}

  \begin{claim}
    $K_{y}^{z_1z_2} \subset Q_{yz_1}$.
    Hence
    there exists an irreducible component
    $K_y^{z_1}$ of $Q_{yz_1} \cap \TT_yQ_{yz_1} \subset \lin{Q_{yz_1}}$
    such that $K_y^{z_1} = K_{y}^{z_1z_2}$ for general $z_2 \in Y$.
  \end{claim}
  \begin{proof}
    Suppose $K_{y}^{z_1z_2} \not\subset Q_{yz_1}$.
    Let $z,w \in Y$ be general
    such that $w \not\in Q_{yz}$.
    For general $\tilde{y} \in Q_{yz}$,
    we have $Q_{\tilde{y}z} = Q_{yz}$.
    
    Since $Y$ is the closure of
    $\bigcup_{\tilde{y} \in Q_{yz}\gen} Q_{\tilde{y}w}$, we have that
    $\Cone_wY$ is the closure of $\bigcup_{\tilde{y} \in Q_{yz}\gen} \lin{Q_{\tilde{y}w}}$.
    Since
    $\tilde y \in K_{\tilde y}^{zw}$,
    we have
    ${Q_{yz}} \subset \overline{\bigcup_{\tilde{y} \in Q_{yz}\gen} K_{\tilde y}^{zw}}$.
    Moreover, since $Q_{yz}$ is codimension one in $\lin{Q_{yz}}$,
    and since $K_{\tilde y}^{zw} \not\subset Q_{yz}$ and $K_{\tilde y}^{zw} \subset \lin{Q_{yz}}$,
    we have
    \[
    \lin{Q_{yz}} \subset \overline{\bigcup_{\tilde{y} \in Q_{yz}\gen} K_{\tilde y}^{zw}}.
    \]
    Since the right hand side is contained in
    $\overline{\bigcup_{\tilde{y} \in Q_{yz}\gen} \lin{Q_{\tilde{y}w}}}$,
    it holds $\lin{Q_{yz}} \subset \Cone_wY$.
    Then $\Cone_w\lin{Q_{yz}} \subset \Cone_wY$,
    where the left hand side is a $4$-plane
    and the right hand side is of dimension $4$.
    It holds that $\Cone_wY$ is a $4$-plane,
    a contradiction.

    Note that, since $K_{y}^{z_1z_2}$ is contained in $Q_{yz_1} \cap \TT_yQ_{yz_1} \subset \lin{Q_{yz_1}}$,
    there exists an irreducible component of $Q_{yz_1} \cap \TT_yQ_{yz_1}$
    which is equal to $K_{y}^{z_1z_2}$ for general $z_2 \in Y$. Hence the latter statement holds.
  \end{proof}

  \begin{claim}
    $\lin{Y} \subset \PP^6$ is of dimension $5$.
  \end{claim}
  \begin{proof}
    Let $y, z_1 \in Y$ be general.
    Since $Y$ is the closure of $\bigcup_{z_2 \in Y\gen} Q_{yz_2}$,
    it follows that $\Cone_yY$ is the closure of $\bigcup_{z_2 \in Y\gen} \lin{Q_{yz_2}}$.
    Since $K_y^{z_1} = K_{y}^{z_1z_2} \subset \lin{Q_{yz_2}}$ for general $z_2 \in Y$,
    we have that $\Cone_yY$ is a cone with vertex $K_y^{z_1} = \PP^{1}$.

    Let $\tilde y, \tilde z_1 \in Y$ be general.
    We may assume
    $\tilde y \notin K_y^{z_1}$ and $y \notin K_{\tilde y}^{\tilde z_1}$.
    We show the statement in the following two steps.
    \\    

    \begin{inparaenum}[\noindent\itshape{}Step 1.]
    \item
      Suppose that two lines $K_y^{z_1}$ and $K_{\tilde{y}}^{\tilde{z_1}}$ intersect at a point $v$.
      Then, for general $s, t \in Y$, the line $K_s^t$ also intersects with each of $K_y^{z_1}$ and $K_{\tilde{y}}^{\tilde{z_1}}$.
      If $v \notin K_s^t$, then
      $s \in K_s^t \subset \lin{K_y^{z_1}, K_{\tilde{y}}^{\tilde{z_1}}}$;
      hence we have
      $Y \subset \lin{K_y^{z_1}, K_{\tilde{y}}^{\tilde{z_1}}} = \PP^2$,
      a contradiction.

      If $v \in K_s^t$, then since $s \in Y$ is general and $\overline{sv} = K_s^t \subset Y$, it follows that $Y$ is a cone with vertex $v$;
      hence $Y \subset \TT_vX$, which implies $\lin{Y} = \TT_vX = \PP^5$.
      \\

    \item
      Suppose $K_y^{z_1}  \cap K_{\tilde{y}}^{\tilde{z_1}} = \emptyset$.
      We have $K_{y}^{z_1} = K_{y}^{z_1\tilde y} \subset Q_{y\tilde y} \subset \lin{Q_{y\tilde y}}$.
      In the same way, $K_{\tilde{y}}^{\tilde{z_1}} \subset Q_{y\tilde y}$.
      Since
      $K_{\tilde y}^{\tilde z_1} = \lin{Q_{\tilde y \tilde z_1}} \cap \lin{Q_{y\tilde y}}$,
      we have
      \[
      K_y^{z_1} \cap \lin{Q_{\tilde y \tilde z_1}} = K_y^{z_1}  \cap K_{\tilde{y}}^{\tilde{z_1}} = \emptyset.
      \]
      
      For the linear projection $\pi_y: \PP^6 \dashrightarrow \PP^5$,
      we consider $\pi_{y}(Y) = \pi_{y}(\Cone_yY)$, a cone with vertex $w:= \pi_y(K_y^{z_1})$.
      Since $w \not\in \pi_y(\lin{Q_{\tilde y \tilde z_1}})$
      and since $\pi_y(Q_{\tilde y\tilde z_1})$       is of codimension $1$ in $\pi_y(Y)$, it follows that $\pi_y(Y)$ is a cone of the quadric $\pi_y(Q_{\tilde y\tilde z_1})$ with vertex $w$.
      Then $\lin{\pi_y(Y)} = \PP^4$.
      Since $y \in Y$ is general,
      it follows that $Y$ is a $3$-fold of degree $3$ in $\lin{Y} = \PP^{5}$.
    \end{inparaenum}
  \end{proof}

  Let us complete the proof. By the above claim, $\lin{Y} = \PP^5$.
  Take $X' := X \cap \lin{Y}$.
  Since the Gauss map $\gamma = \gamma_X: X \rightarrow \Pv[6]$ is finite,
  $X'$ is singular at most finitely many points
  ($X'$ is singular at $x$ if and only if $\gamma(x) = \lin{Y}$
  in $\Pv[6]$).
  In particular, $X'$ is irreducible.
  (This is because, if $X' = X'_1 \cup X'_2 \subset \lin{Y} = \PP^{5}$,
  then $X'_1 \cap X'_2 \subset \Sing X'$.)
  Here $Y \subset X' \subset \lin{Y}$.

  Take a general hyperplane $M \subset \lin{Y} = \PP^{5}$
  such that $X'' := X' \cap M$ is smooth.
  Then we have
  \[
  Y \cap M \subset X'' \subset M = \PP^4,
  \]
  where $Y \cap M$ is a surface of degree $\leq 4$ as in \autoref{degY34}
  and $X''$ is a smooth $3$-fold of degree $6$ in $\PP^4$.
  This is a contradiction
  since $\Pic(\PP^{4}) \simeq \ZZ \rightarrow \Pic(X''): \sO_{\PP}(1) \rightarrow \sO_{\PP}(1)|_{X''}$
  is isomorphic due to the Lefschetz theorem.
\end{proof}

\begin{ex}\label{thm:ex-Fermat}
  Let $X \subset \PP^7$ be the Fermat hypersurface of degree $6$.
  Then $\PP^3 \subset X$.
  Thus
  $R_2(\PP^3) \subset R_2(X)$ is of dimension $8$,
  and the expected dimension of $R_2(X)$ is $3n-2d-2 = 7$.
  For a general hyperplane $\PP^6 \subset \PP^7$,
  we have $\PP^2 \subset X_1 := X \cap \PP^6$.
  Then
  $R_2(\PP^2) \subset R_2(X_1)$ is of dimension $5$
  and the expected dimension of $R_2(X_1)$ is $4$.

  In these examples,
  each $C \in R_2(\PP^3)$ (resp. $C \in R_2(\PP^2)$)
  satisfies $\lin{C} \subset \PP^3 \subset X$ (resp. $\lin{C} = \PP^2 \subset X_1$).
\end{ex}

\begin{ex}\label{thm:ex-sub-Cone}
  Let $X \subset \PP^{10}$ be the smooth hypersurface of degree $10$
  defined by the following polynomial,
  \[
  f := x_0^{8}(x_0^2 + x_1^2 + x_2^2)
  + \sum_{i=1}^5 x_i^{10} - \sum_{i=1}^5 x_{i+5}^{10}.
  \]
  Then the expected dimension of $R_2(X)$ is $8$.
  We consider the $5$-plane
  \[
  M := \bigcap_{i=1}^{5} (x_i - x_{i+5} = 0) = \PP^5 \subset \PP^{10},
  \]
  and take
  $Y \subset X \cap M$,
  to be the zero set of
  $x_0^2 + x_1^2 + x_2^2$
  in $M$.
  Since $Y$ is a cone of the conic $(x_0^2 + x_1^2 + x_2^2 = 0) \subset \PP^2$, we have a birational map
  $R_2(Y) \rightarrow \Gr(2, \PP^5): C \mapsto \lin{C}$.
  In particular,
  $\dim R_2(Y) = 9$.
  Thus, for an irreducible component
  $R \subset R_2(X)$ containing $R_2(Y)$,
  the dimension of $R$ is greater that the expected dimension.
\end{ex}

\end{document}